\documentclass{article}
\usepackage{amssymb,amsmath}
\usepackage{graphicx,color}
\def\bigO(#1){\ensuremath \mathcal{O}(#1)}
\newcommand{\N}{ {\mathbb N} }
\newcommand{\R}{ {\mathbb R} }
\newcommand{\C}{ {\mathbb C} }
\newcommand{\DD}{\mathcal{D}}
\newcommand{\LL}{\mathcal{L}}
\newcommand{\PP}{\mathcal{P}}
\newcommand{\MG}{\mathcal{M}} 

\newcommand{\EE}{ {\mathbb E} }
\newcommand{\MM}{ {\mathbb M} }
\newcommand{\Mb}{ {\bf M}}
\newcommand{\Zh}{\hat Z}
\def\ones{\mbox{\normalfont{1\hskip-0.24em l}}}
\def\diag{\mbox{diag}}
\newcommand{\parn}{\par\noindent}

\newcommand{\T}{^{\sf T}}
\newcommand{\gmm}{g}\newcommand{\Gmm}{G}
\newcommand{\nrad}[2]{r\big(#1,#2\big)}
\newenvironment{proof}{{\bf Proof.}}{\mbox{$\quad\square$}}

\newtheorem{theorem}{Theorem}[section] 
\newtheorem{lemma}[theorem]{Lemma}

\newtheorem{definition}[theorem]{Definition}
\newcommand{\example}{\refstepcounter{theorem}
 \par\noindent{\bf Example \thetheorem{} }}

\newtheorem{remark}[theorem]{Remark}
\numberwithin{equation}{section}

\begin{document}
\title{Algebraic criteria for A-stability of peer two-step methods}
\author{Bernhard A. Schmitt\thanks{Fachbereich Mathematik und Informatik, Philipps-Universit\"at, D-35032 Marburg, Germany
 ({\tt schmitt@mathematik.uni-marburg.de}).}}
\date{}
\maketitle
\begin{abstract}
A new criterion for A-stability of peer two-step methods is presented which is verifiable exactly in exact arithmetic by checking semi-definite\-ness of a certain test matrix.
It depends on the existence of two positive definite weight matrices for a given method.
Although the initial approach is different using properties of the numerical radius the criterion itself resembles the one from algebraic stability of General Linear Methods.
Known numerical algorithms for the computation of the unknown weight matrices suffer from rank deficiencies of the test matrix.
For $s$-stage peer methods of order $s-1$ this rank defect is identified with an explicit block diagonal decomposition of the test matrix in trivial and definite blocks.
In the design of methods its coefficients are unknown and an explicit parametrization of A-stable peer methods of order $s-1$ is presented with a weight matrix as parameter. 
This leads to a general existence result for any number of stages.
The restrictions for efficient L-stable peer methods like diagonally-implicit and parallel ones are also discussed and such methods with $3$ and $4$ stages are constructed.
\end{abstract}


\section{Introduction}
A-stability requires that numerical solutions for the Dahlquist test equation
\begin{align}\label{testeq}
 y'(t)=\lambda y(t),\quad \lambda\in\C_-:=\{z\in\C:\, \Re\lambda<0\},
\end{align}
be uniformly bounded for $t\ge0$.
Although this concept seems rather limited addressing only the simple scalar equation it gains greater significance by the theorem of von Neumann \cite{vNeum51}, \cite[Th.IV.11.2]{HaWe96} which is applicable to A-stable one-step discretizations of linear systems $y'(t)=Ay(t)$ if $A$ has non-positive logarithmic norm, e.g. \cite{CGPM00}.
For Runge-Kutta methods the stability function $R:\,\C\mapsto\C$ is a scalar function.
Establishing A-stability here is possible through the maximum principle by considering the E-polynomial $|Q(i\eta)|^2-|P(i\eta)|^2,\eta\in\R$, of $R(z)=P(z)/Q(z)$ on the imaginary axis if $R$ has no poles in $\C_-$, \cite{HaWe96}.
This is an analytic criterion which may be difficult to check exactly for higher degrees.
However, Hairer, T\"urke and Scherer also derived a purely algebraic criterion for A-stability of Runge-Kutta methods depending on the semi-definiteness of a certain matrix, \cite{HaTu84}, \cite{SchTu89}.
It is a weaker version of the criterion for algebraic stability since the diagonal matrix of quadrature weights may be replaced by an unknown semi-definite weight matrix.
For $s$-stage peer two-step methods the stability function is a matrix function $\Mb:\,\C\mapsto\C^{s\times s}$.
In this case we define A-stability (i.e. A-acceptability of the stability matrix) by
\begin{align}\label{Aacc}
 \varrho\big(\Mb(z)\big)<1,\quad \lambda\in\C_-,
\end{align}
where $\varrho$ denotes the spectral radius.
A matrix valued version of the von-Neumann theorem is due to Nevanlinna, \cite{Nevan85}.
It uses assumptions on the spectral norm $\|\Mb(z)\|_2$ not on the spectral radius.
However, it will be shown that such bounds follow from the new algebraic criterion.
\par
Peer two-step methods were introduced in \cite{ScWe04} in a linearly-implicit version.
They are time stepping schemes using $s$ stages $Y_{mi}\in\R^n$, $i=1,\ldots,s$, on time intervals $[t_m,t_{m+1}]$ with general stepsizes $h_m=t_{m+1}-t_m$.
However, the discussion of A-stability is restricted to a fixed stepsize $h_m\equiv h$, here.
The stages $Y_{mi}$ are associated with the ODE solution at off-step points $t_{mi}=t_m+hc_i$ with fixed nodes $c_i$, $i=1,\ldots,s$.
For an autonomous problem $y'(t)=f\big(y(t)\big)$ with right-hand side $f:\,\R^n\to\R^n$ and function evaluations $F_{mi}=f(Y_{mi})$ the peer two-step method is given by
\begin{equation}\label{PeerV}
Y_{mi}-h_m\sum_{j=1}^i \gmm_{ij}F_{mj}
=\sum_{j=1}^s b_{ij}Y_{m-1,j}+h_m\sum_{j=1}^sa_{ij}F_{m-1,j},\ i=1,\ldots,s.
\end{equation}
The elements $a_{ij},b_{ij},\gmm_{ij}$, $i,j=1,\ldots,s$ are the coefficients of the method and will be assembled in square matrices $A=(a_{ij}),\,B=(b_{ij}),\,\Gmm:=(\gmm_{ij})\in\R^{s\times s}$.
For a constant stepsize these matrices are independent of the time step.
For general dimension $n\in\N$ stage vectors and function evaluations may be collected in matrices $Y_m\T=(Y_{m1},\ldots,Y_{ms})\in\R^{s\times n}$, $F_m\T=\big(f(Y_{mi})\big)_{i=1}^s$ yielding a compact representation of the method (\ref{PeerV}) as
\begin{align}\label{PeerS}
  Y_m-h_m\Gmm_m F_m=B_mY_{m-1}+h_mA_mF_{m-1}.
\end{align}
The structure of the matrix $\Gmm_m$ determines the numerical effort of the scheme.
In order to keep the effort of implicit methods moderate diagonally-implicit methods with a lower triangular matrix $\Gmm$ may be considered \cite{BWPS10} or parallel methods where $\Gmm$ is diagonal, \cite{ScWe04,SWE05,MPPW}.
For nonlinear ODEs \eqref{PeerS} is an implicit scheme.
The notion 'peer' refers to the fact that all stages $Y_{mi}$ have the same order.
Hence, a polynomial predictor is available to obtain a linearly-implicit method of full accuracy by performing only one Newton step, \cite{MPPW}.
Finally we note that peer methods are a special subclass of General Linear Methods (GLMs) where all internal stages are passed to the next interval.
\par
Verification of A-stability for GLMs is more difficult than for Runge-Kutta and multistep methods.
Floating point computations of eigenvalues are not fully reliable since $\varrho(\Mb(0))=1$ by preconsistency (see below).
The Schur criterion for polynomials may be used, see \cite{Jack09,BraJa14}, requiring global bounds on certain polynomials on $\C_-$.
Verification may also be difficult here for higher degrees.
In this paper a sufficient criterion is presented where only semi-definiteness of a certain $(2s)\times(2s)$-matrix $\MM$ (called a 'test matrix') has to be checked.
This test can be performed exactly e.g. by using rational arithmetic and a Cholesky-type decomposition if the matrix elements are rational.
The criterion is based on a reformulation of the eigenvalue problem for $\Mb(z)$ and uses semi-definiteness in an inner product described by an unknown symmetric positive definite matrix $Z$ (called a weight matrix).
Using well-known properties of the numerical radius of matrices an equivalent characterization by semi-definiteness of a certain block matrix $\MM$ depending on two unknown matrices $Z$ and $W$ is obtained.
This criterion, which is a feasibility problem of semidefinite programming (\cite{Vdboy96}), is related to a reformulation of algebraic stability presented in \cite{SchmAS12}.
As in the case of Runge-Kutta methods the restrictions on the unknown weight matrices $Z,W$ are weaker for A-stability.
The construction of such feasible matrices for given algebraically stable GLMs has been discussed by Hewitt and Hill \cite{HeHi09,Hill10} and by the author \cite{SchmAS12}.
The corresponding numerical methods were based on the solution of algebraic Riccati equations and certain symplectic eigenvalue problems.
However, computations were notoriously ill-conditioned due to singular Jacobians or higher multiplicity of eigenvalues.
The reason for these numerical difficulties is an inherent rank deficiency of the positive semidefinite test matrix $\MM$.
A rank defect one is well-known being due to the preconsistency condition, \cite{BurBut80}, \cite{HaWe96}.
This rank defect increases for higher order peer methods.
By an explicit transformation this rank deficiency is identified here exactly and an explicit block decomposition $\MM=0\oplus\MM_D$ is provided where the lower block $\MM_D$ can be positive definite.
Positive definiteness can be checked reliably by numerical computations.
We expect that this background may be beneficial also in the context of other General Linear Methods.
In this paper no numerical methods for finding A-stable methods are discussed, with one exception methods are constructed by formal computations with Maple.
\par
Much of the structure exploited in our analysis is due to the requirement of higher order.
The preconsistency condition for peer methods simply reads
\begin{align}\label{precon}
 B\ones=\ones,\quad \ones:=(1,\ldots,1)\T,
\end{align}
leading to one eigenvalue $1$ of $B$.
For $s$-stage peer methods order $s-1$ is easily obtained for arbitrary stepsize sequences by satisfying certain matrix equations.
For stiff problems so-called stiffly-accurate peer methods \eqref{PeerS} containing no explicit terms due to the choice $A=0$ have been successfully used in several papers, e.g. \cite{MPPW,BWPS10}.
For such methods and constant stepsize the conditions for order $s-1$ correspond to the choice (\cite{MPPW})
\begin{align}\label{ordsme}
  B=(I-\Gmm E)\Theta.
\end{align}
The matrices $E$ and $\Theta$ correspond to polynomial differentiation and extrapolation at the off-step points $t_{ki}=t_k+hc_i$. $i=1,\ldots,s$, $k=m,m+1$.
With the aid of the Vandermonde matrix $V=\big(c_i^{j-1}\big)\in\R^{s\times s}$, the shift matrix $F_0=\big(\delta_{i,j-1}\big)$ and the scaling matrix $D=\diag_i(i)$ they are given by $E=VDF_0\T V^{-1}$ and $\Theta=VPV^{-1}$.
The well-known relations $P=\exp(DF_0\T)$ \cite{MPPW} and
\begin{align}\label{expE}
  \Theta=e^E=\exp(E)
\end{align}
between both are discrete versions of the Taylor theorem and will play a crucial role later on.
Of course, the preconsistency condition \eqref{precon} follows from \eqref{ordsme} since $V^{-1}\ones=e_1$ is the first unit vector.
By choosing the nodes and the matrix $\Gmm$ a stiffly-accurate peer method of order $s-1$ is specified completely.
On general grids zero stability, i.e. stability of the scheme applied to the trivial ODE $y'=0$, may be a difficult requirement for peer methods, see \cite{MPPW}.
For constant stepsize it requires that the powers of $B$ are uniformly bounded. This leads to the condition $\varrho(B)\le1$ where eigenvalues on the unit circle are non-defective.
\par
Positive definiteness of symmetric matrices is denoted by $Z\succ0$ or 
$Z\in\DD^s:=\{X\in\R^{s\times s}:\,X\T=X\succ0\}$, semi-definiteness by $Z\succeq0$ or $Z\in\DD_0^s:=\{X\in\R^{s\times s}:\,X\T=X\succeq0\}$.
Only real matrices will be considered but vectors may be complex $x\in\C^s$, and $x^\ast=\bar x\T$ denotes their conjugate transpose.
In Section~\ref{Salgbed} we present the basic approach for establishing A-stability through a definiteness argument depending on an unknown weight matrix $Z\in\DD^s$.
It leads to a criterion for the generalized numerical radius of two matrices.
Using a result of Ando an equivalent characterization by semi-definiteness of a $(2s)\times(2s)$ test matrix $\MM$ is obtained which depends on a second matrix $W\in\DD_0^s$.
In Section~\ref{Sstac} we concentrate on methods with order $s-1$ and present a transformation revealing the rank deficiency of the test matrix explicitly.
A careful analysis of the requirements leads to an explicit relation between the unknown weight matrices $W$ and $Z$ and an existence result for A-stable peer methods for arbitrary $s\in\N$.
Also, the focus is moved from the existence of weight matrices to the existence of methods and a parametrization of A-stable methods is obtained where one of the weight matrices $W$ or $Z$ is a parameter.
The construction of practical methods in Sections~4 and 5 uses different techniques.
For diagonally-implicit methods triangular decompositions and a certain triangular canonical form of matrices are used.
The design of parallel methods with diagonal $G$ is based on a representation of $V^{-1}GV$ from \cite{SWE05}.
In both cases A-stable example methods are constructed for $s=3$ and $4$.
%
\section{A-stability and the numerical radius}\label{Salgbed}
Applying the peer method \eqref{PeerS} to the test equation \eqref{testeq} yields the recursion $Y_m=\Mb(z)Y_{m-1}$, where $\Mb(z)$ is the stability matrix of the peer method given by
\begin{align}\label{Stabmat}
 \Mb(z)=(I-z\Gmm)^{-1}(B+zA),\quad z\in \C.
\end{align}
Obviously, $\Mb(0)=B$ and stiffly accurate methods with nonsingular coefficient $\Gmm$ satisfy $\Mb(\infty)=0$ since $A=0$.
Hence, if such methods are A-stable they are also L-stable.
Due to the simple product form a symmetry between the eigenvalues $\lambda$ and the parameter $z$ may be used in the formulation which is well known from multistep methods.
Obviously, for $\det(I-z\Gmm)\not=0$ some $\lambda\in\C$ is an eigenvalue of $\Mb(z)$ iff
\begin{align}\label{zlambda}
 \det(B+zA+\lambda zG-\lambda I)=0.
\end{align}
This condition may be considered as an equation either for $\lambda$ or $z$ with solutions $\lambda(z)$ or $z(\lambda)$.
Hence, A-stability may be characterized in the following two equivalent ways where for each pair $(z,\lambda)$ satisfying \eqref{zlambda} holds
\begin{align}
\begin{array}{lll}
  \Re z<0&\Rightarrow& |\lambda|<1,\\
  |\lambda|\ge1&\Rightarrow& \Re z\ge0.
\end{array} 
\end{align}
Moving $\lambda$ along the unit circle and solving for $z$ gives the root-locus curves defining the stability regions of multistep methods, \cite[Sect.V.1]{HaWe96}.
A-stability means that these curves do not intersect the open left complex halfplane.
Considering $z$ as a solution of equation \eqref{zlambda} and $\lambda$ as a parameter leads to the generalized eigenvalue problem
\begin{align}\label{zewp}
 z(G+\frac1{\lambda}A)x =(I-\frac1{\lambda}B)x,\quad x\not=0,
\end{align}
with eigenvalue $z=z(\lambda)$.
For A-stable methods and $|\lambda|\ge1$ only solutions $z$ with non-negative real part are possible.
This property may be implied by a definiteness condition with respect to a Hermitian, positive definite matrix $Z$.
In the most general form of the following argument this weight matrix $Z$ might depend on $\lambda$, i.e. $Z=Z(\lambda)=Z(\lambda)^\ast$.
However, in this paper we consider only a constant, real weight matrix $Z\in\DD^s$.
Thus, using some nontrivial $x\in\C^s$, and $|\lambda|\ge1$, equation \eqref{zewp} is multiplied by $x^\ast\big(G+(1/\bar\lambda)A\big)\T Z$.
The real part of the result is
\begin{align}\notag
 2(\Re z)\|Z^{1/2}(G+\frac1{\lambda}A)x\|_2^2
  =&2\Re x^\ast\big(G+\frac1{\lambda}A\big)^\ast Z (I-\frac1{\lambda}B)x\\\label{Zdefnt}
  =& x^\ast\big(G\T Z+ZG-\frac1{|\lambda|^2}(A\T ZB+B\T ZA)\big)x\\\notag
   &+2\Re\frac1{\lambda} x^\ast\big(ZA-G\T ZB\big)x.
\end{align}
Hence, if the right-hand side of this equation is non-negative for any $\lambda$ with $|\lambda|\ge1$, $\Re z$ must be so, too.
In fact, this condition only depends on the absolute value $\mu=1/|\lambda|\le1$.
\begin{lemma}\label{Lnradb}
Let $Z\in\DD^s$ be such that $Q_\mu:=\frac12(G\T Z+ZG-\mu^2(A\T ZB+B\T ZA))\succ0$, $\mu\in[0,1]$.
Then, a sufficient condition for A-stability of the peer method \eqref{PeerS} is
\begin{align}\label{vanura}
\mu\frac{|x^\ast(ZA-G\T ZB)x|}{x^\ast Q_\mu x}\le 1\quad \forall x\in\C^s\setminus\{0\},\ \mu\in[0,1].
\end{align}
\end{lemma}
\begin{proof}
For $\alpha\in\C$ it holds that $\max\{\Re(\frac{\alpha}{\lambda}):\,|\lambda|=1/\mu\}=\mu|\alpha|$.
Hence, requiring non-negativity of the right-hand side of \eqref{Zdefnt} with fixed $x\not=0$ for all $|\lambda|=1/\mu$ means
\begin{align}\notag
 &\min_{\mu|\lambda|=1}
  \big(x^\ast\big(G\T Z+ZG-\frac1{|\lambda|^2}(A\T ZB+B\T ZA)\big)x \\\notag
  &\qquad +2\Re\frac1{\lambda} x^\ast\big(ZA-G\T ZB\big)x\big)\\\label{qfsdef}
   =&x^\ast\big(G\T Z+ZG-\mu^2(A\T ZB+B\T ZA)\big)x
   -2\mu|x^\ast\big(ZA-G\T ZB\big)x|\stackrel!\ge0.
\end{align}
The statement follows by considering arbitrary $x\not=0$ and $\mu\in[0,1]$ with $Q_\mu\succ0$:
\[ \max_{x\not=0}\frac{2\mu|x^\ast\big(ZA-G\T ZB\big)x|}{x^\ast\big(G\T Z+ZG-\mu^2(A\T ZB+B\T ZA)\big)x} \le1\ \forall \mu\in[0,1].
\]
\end{proof}
\parn
The inequality \eqref{vanura} is a condition on the numerical radius of a matrix pencil.
\begin{definition}
Let $U\in \R^{s\times s}$ and $N\in \DD^{s}$.
The (generalized) Rayleigh quotient of the matrix pencil $(U,N)$ is defined for $x\in\C^s\setminus\{0\}$ by $x^\ast Ux/x^\ast Nx$.
The {\em numerical range} (or {\em field of values}) of the pencil is the set of all Rayleigh quotients, $\{x^\ast Ux/x^\ast Nx:\, x\in\C^s\setminus\{0\}\}$.
The absolutely largest element in the numerical range is denoted as the numerical radius
\begin{align}\label{numrad}
  \nrad{U}{N}:=\max\big\{\big|\frac{x^\ast Ux}{x^\ast Nx}\big|:\ x\in\C^s\setminus\{0\}\big\}.
\end{align}
\end{definition}
\begin{remark}
a) The numerical range is a convex set containing all eigenvalues of the generalized eigenvalue problem $Ux=\lambda Nx$, \cite{HoJo85}.
\\{}
b) For the matrix pencil $U-\lambda I$ the matrix $N=I$ is omitted in the notation, $r(U):=\nrad{U}{I}$.
The general form is only a simple extension of $r(U)$ since any general $N\in\DD^s$ possesses a Cholesky decomposition $N=LL\T$ and it holds that $\nrad{U}{N}=r(L^{-1}U{L^{-T}})$ using the shorthand notation ${L^{-1}}\T=L^{-T}$.
The numerical radius $r(U)$ has some interesting properties.
It is a matrix norm bounded by the spectral norm, $\varrho(U)\le r(U)\le\|U\|_2$.
It is not sub-multiplicative but it still satisfies the power inequality $\frac12\|U^k\|_2\le r(U^k)\le r(U)^k\,\forall k\in\N$.
Hence, $r(U)\le1$ ($<1$) implies power boundedness (contractivity) of $U$, \cite{HoJo85}.
\end{remark}
\par
The dependence on the parameter $\mu$ in Lemma~\ref{Lnradb} makes the discussion difficult.
However, for stiffly accurate peer methods with coefficient $A=0$ it is obvious that $\mu=1$ is the critical value in \eqref{vanura}.
With the exception of the first paper on peer methods \cite{ScWe04} only such methods have been considered in the literature for the solution of stiff problems due to their superior damping of stiff solution components.
For $A=0$ the condition \eqref{vanura} for A-stability can be re-stated with the numerical radius as
\begin{align}\label{vanurd}
 r(G\T ZB,Q)\le1,\quad Q:=\frac12(G\T Z+ZG)\succ0.
\end{align}
If this condition is satisfied it will be with equality for any practical peer method due to the preconsistency condition \eqref{precon}.
In the Rayleigh quotient the preconsistency vector $\ones$ gives the nontrivial denominator $\ones\T Q\ones=\ones\T ZG\ones>0$ and the  nominator $|\ones\T G\T ZB \ones|=|\ones\T ZG\ones|$ is the same.
This critical vector $\ones$ leads to tight inequalities in the new criterion for A-stability similar to those in algebraic stability \cite{HaWe96}.
\par
The condition \eqref{vanurd} is not very practical yet since the computation of the numerical radius is difficult.
However, there exists a purely algebraic criterion due to Ando, see also \cite{Math93}.
\begin{lemma}[\cite{Ando73}]\label{LAndo}
For any matrix $U\in\R^{s\times s}$ the condition $r(U)\le1$ is equivalent with the existence of a symmetric matrix $W\in\R^{s\times s}$ such that
\begin{align}\label{RAndo}
 \begin{pmatrix} 2I-W&U\\ U\T&W \end{pmatrix}\succeq 0.
\end{align}
\end{lemma}
The statement differs slightly from the original one which used a matrix $W'=I-W$.
The version \eqref{RAndo} is more appropriate here and reveals that $W$ itself is semi-definite, $W\in\DD_0^s$.
We also note that semi-definiteness in \eqref{RAndo} with the choice $W=I$ corresponds to the stronger condition $\|U\|_2\le1$.
Lemma~\ref{LAndo} yields the following purely algebraic condition for A-stability of peer methods with $A=0$.
Actually, it uses slightly weaker assumptions than Lemma~\ref{Lnradb}.
\begin{theorem}\label{TAAStb}
If there exist matrices $Z\in\DD^s$ and $W\in\DD_0^s$ such that 
\begin{align}\label{MDef}
 \MM:=\begin{pmatrix}
  G\T Z+ZG-W&-G\T ZB\\
  -B\T ZG&W
 \end{pmatrix}\succeq0,
\end{align}
then, the stiffly accurate peer method \eqref{PeerV} with $A=0$ is A-stable.
\end{theorem}
\begin{proof}
Since the matrix $Q$  may not be definite, Lemma~\ref{LAndo} cannot be used directly.
However, considering the quadratic form of the semidefinite matrix $\MM$ with vectors of the form $(x\T,\alpha x\T)\T$, $x\in\C^s$, $\alpha\in\C$, gives
\begin{align*}
 0\le& (x^\ast,\bar\alpha x^\ast)\MM\begin{pmatrix}x\\\alpha x \end{pmatrix}\\
 =&2x^\ast Qx-2\Re\big(\alpha x^\ast G\T ZB x\big)+(|\alpha|^2-1) x^\ast W x
\end{align*}
Choosing $\alpha$ such that $|\alpha|=1$ and $\Re(\alpha x^\ast G\T ZBx)=|x^\ast G\T ZBx|$ the condition is identical with \eqref{qfsdef} showing A-stability.
\end{proof}
\begin{remark}
For the unknown matrices $Z$ and $W$ the inequality \eqref{MDef} is a linear feasibility problem of semi-definite programming, see \cite{Vdboy96}.
For any feasible solution $(Z,W)$ also all positive multiples are solutions and some normalization may be appropriate.
\end{remark}
With known matrices $Z,W$ verification of A-stability by Theorem~\ref{TAAStb} requires only the computation of a Cholesky-type decomposition of $\MM$.
Hence, for rational or algebraic entries exact verification is possible.
This is an advantage over other criteria from the literature like the Schur criterion \cite{Jack09} requiring that certain rational functions be positive or bounded by one on the whole complex halfplane.
Exact verification may be crucial since the test matrix $\MM$ is rank deficient and semi-definiteness may not be decidable with floating point arithmetic.
\begin{lemma}\label{LZWeins}
Under the assumption \eqref{precon}, $B\ones=\ones$, the matrix $\MM$ is rank deficient.
\end{lemma}
\begin{proof}
$\MM$ has a nontrivial kernel, its quadratic form with $(\ones\T,\ones\T)\T$ vanishes.
\end{proof}
\parn
We note without proof that similar to \cite{SchmAS12}, \cite[Sect.I.9]{HaWe96} the following necessary conditions on $W$ and $Z$ may be derived, 
\begin{align*}
 (B\T-I)ZG\ones=0,\quad
  W\ones=B\T ZG\ones.
\end{align*}
Although the initial approach \eqref{qfsdef} leading to the A-stability criterion was independent from the concept of algebraic stability the final criterion from Theorem~\ref{TAAStb} obtained for a constant weight matrix $Z$ is strongly related.
By using Sylvester's law of inertia and special congruence transformations it was shown in \cite{SchmAS12} that algebraic stability of stiffly accurate peer methods is essentially equivalent with the existence of a diagonal matrix $D$ and a weight Matrix $W\in\DD^s$ such that
\begin{align}\label{algstab}
\begin{pmatrix}
 G\T D+DG-G\T WG&G\T DG^{-1}B\\
 B\T G^{-T}DG &W
\end{pmatrix}\succeq0.
\end{align}
Assuming a nonsingular coefficient $G$ and introducing the matrix $Z:=G^{-T}DG^{-1}$ this block matrix is congruent with $\MM$ from \eqref{MDef}.
Hence, the main difference between algebraic stability and the sufficient criterion for A-stability is that the weight matrix $Z$ is restricted by the requirement that $G\T ZG=D$ is diagonal.
In algebraic stability also the conditions on $W$ and $Z$ have changed slightly, $W\in\DD^s$ and $D\in\DD_0^s$.
Hence, the results from the next sections may also be of interest in this larger context. 
In fact, by combining the congruence transformation from \cite{SchmAS12} with the original proof from \cite{BurBut80} the following norm bound for the stability matrix $\Mb(z)$ \eqref{Stabmat} of stiffly-accurate methods is obtained.
It verifies the assumption of the Nevanlinna theorem \cite{Nevan85}.
\begin{lemma}\label{LNrmM}
If there exist $W,Z\in\DD^s$ such that $\MM\succ0$ in \eqref{MDef}, then for the stability matrix $\Mb(z)=(I-zG)^{-1}B$ it holds that
\[ \|W^{1/2}\Mb(z)W^{-1/2}\|_2\le1\ \forall z\in\C_-.\]
\end{lemma}
\begin{proof}
With the variables from the time step $Y_m-BY_{m-1}=z\Gmm Y_m\in\C^s$ applied to \eqref{testeq} the quadratic form of the real matrix $\MM$ is nonnegative,
\begin{align*}
 0\le&(Y_m^\ast,Y_{m-1}^\ast)
 \begin{pmatrix} G\T Z+ZG-W&-G\T ZB\\ -B\T ZG &W \end{pmatrix}
 \begin{pmatrix}Y_m\\ Y_{m-1}\end{pmatrix}\\
 =&Y_{m-1}^\ast WY_{m-1}-Y_{m}^\ast WY_{m}+2\Re\big(Y_m^\ast G\T Z(Y_m-BY_{m-1})\big)\\
 =&Y_{m-1}^\ast WY_{m-1}-Y_{m}^\ast WY_{m}+2\Re(z)Y_m^\ast G\T Z\Gmm Y_m.
\end{align*}
Since $\Re z<0$ and $Y_m^\ast G\T Z\Gmm Y_m\ge0$ the assertion follows by $\|W^{1/2}Y_m\|^2=Y_{m}^\ast WY_{m}\le Y_{m-1}^\ast WY_{m-1}=\|W^{1/2}Y_{m-1}\|_2^2$.
\end{proof}
\par
The problem of finding suitable weight matrices $D,W$ for \eqref{algstab} was discussed by Hewitt and Hill \cite{HeHi09} and later by the author \cite{SchmAS12} by considering algebraic Riccati equations.
Hill \cite{Hill10} also used a relation with certain symplectic eigenvalue problems.
However, both approaches suffer from ill-conditioning caused by preconsistency, see Lemma~\ref{LZWeins}, since the Riccati equation has a singular Jacobian in its solutions and the symplectic eigenvalue problem has critical eigenvalues with higher multiplicity.
Hence, numerical iterations may stall or suffer from severe loss of accuracy (\cite{HeHi09}). 
\par
In this paper, we will use a different approach which explicitly reveals the rank defect and also changes the focus of the discussion.
The algebraic criterion will be considered as a tool in the construction of A-stable peer methods.
In this context, not the weight matrices $W,Z$ of a given method are unknown in the first place but the whole method with coefficients $B,G$ and the weights $W,Z$ at the same time.
Hence, the following discussion treats all these matrices as unknowns.
However, additional requirements are introduced in order to narrow the search to methods of practical interest.

\section{Stiffly accurate peer methods of order $s-1$}\label{Sstac}
For methods of higher order the coefficient matrices $B$ and $G$ are coupled.
The relation for order $s-1$  is given by \eqref{ordsme}, $B=\Theta-GE\Theta=(I-GE)\Theta$, where $\Theta$ depends on the off-step nodes $(c_i)$ only.
Thus, the A-stability condition \eqref{MDef} changes for methods of order $s-1$ to
\begin{align}\label{MDefH}
 \MM=\begin{pmatrix}
  G\T Z+ZG-W&-G\T Z(I-GE)\Theta\\
  -\Theta\T(I-GE)\T ZG\T &W
 \end{pmatrix}\succeq0.
\end{align}
This condition is discussed now as a nonlinear feasibility problem for the triple $(G,W,Z)$ of unknown matrices without focusing attention on any particular unknown.
The implicit dependence on the nodes through the Vandermonde matrix $V$ is  considered later.
\subsection{A rank revealing congruence transformation}
Positive semi-definiteness of the block matrix in \eqref{MDefH} requires semi-definiteness of both diagonal blocks, e.g. $W\succeq0$ in the second.
Not much information can be gained for the first diagonal block since it has a nonlinear dependence on all three unknowns $G,W,Z$.
Now, an additional congruence transformation on \eqref{MDefH} is introduced and it is also convenient to use a transformed weight matrix $\hat W:=\Theta^{-T}W\Theta^{-1}$:
\begin{align}\notag
 \hat\MM:=&\begin{pmatrix} I&\Theta^{-T}\\ 0&\Theta^{-T} \end{pmatrix}
  \begin{pmatrix} G\T Z+ZG-W&-G\T Z(I-GE)\Theta\\
  -\Theta\T(I-GE)\T ZG\T &W \end{pmatrix}
  \begin{pmatrix} I&0\\ \Theta^{-1}&\Theta^{-1} \end{pmatrix}
\\\label{MThtr}
 =&\begin{pmatrix}
   (G\T ZG)E+E\T(G\T ZG)+\hat W-\Theta\T\hat W\Theta
   &\hat W-G\T Z+(G\T ZG)E\\
   \hat W-ZG+E\T(G\T ZG)&\hat W \end{pmatrix}.
\end{align}
Before proceeding we note that there is the interesting class of FSAL methods ({\em first same as last}) \cite{SWB12} with a singular coefficient $\Gmm$.
Hence, nonsingularity of $\Gmm$ will be assumed only later on starting with subsection~\ref{SGreg}.
\par
The transformed matrix \eqref{MThtr} is better suited for the analysis than \eqref{MDefH} since the first diagonal block depends linearly on the two matrices $\hat W$ and $\Zh:=G\T ZG$, but not on $G$ alone.
Furthermore, it allows to identify the rank defect of $\MM$ accurately  since the differentiation matrix $E$ is nilpotent and also one is an $s$-fold eigenvalue of the extrapolation matrix $\Theta$.
The two linear maps associated with the first diagonal block are considered in detail, on a matrix $X\in\R^{s\times s}$ they act by
\begin{align}\label{DLLPP}
 \begin{array}{ll}
 \LL_E:&X\mapsto XE+E\T X,\\
 \PP_E:&X\mapsto\Theta\T X\Theta-X,\quad \Theta=e^E,
 \end{array}
\end{align}
see \eqref{expE}. 
Both $\LL$ and $\PP$ map the set of symmetric matrices onto itself. 
Simultaneous similarity transformations of the matrices $E$ and $\Theta$ lead to congruences for these maps.
Indeed, with $X\in\R^{s\times s}$ and nonsingular  $U\in\R^{s\times s}$ we have 
\begin{align}\notag
 U\T\LL_E(X)U=&(U\T X U)(U^{-1} EU)+(U^{-1} EU)\T(U\T XU)\\\label{KongL}
  =&\LL_{U^{-1} EU}(U\T XU),
 \\\label{KongP}
 U\T\PP_E(X)U=&
 \PP_{U^{-1} EU}(U\T XU).
\end{align}
These congruence-similarity transformations will be an important tool in the rest of the paper establishing a similarity between different peer methods.
However, such methods are of practical use only if the matrices $E,\Theta$ have an inherent Vandermonde structure.
This means that they are similar to $V^{-1}EV=\tilde E=DF_0\T$ and $V^{-1}\Theta V=P$ by a real Vandermonde matrix $V$.
Many proofs will be based on the transformed 'Nordsieck' version with the sparse matrix $\tilde E$ and the triangular Pascal matrix $P=\exp(\tilde E)$. 
\par
There is a strong relation between the maps $\LL_E$ and $\PP_E$ which will be helpful.
This relation is easily seen by considering the matrix representations of these maps.
By introducing the $vec$-operator mapping matrices from $\R^{s\times s}$ to vectors from $\R^{s^2}$ by stacking the columns, i.e. $vec(X)=(x_{11},x_{21},\ldots,x_{s-1,s},x_{ss})\T$, these maps may be written with the Kronecker product in the form
\begin{align*}
 vec\big(\LL_E(X)\big)=&(I\otimes E\T+E\T\otimes I)vec(X)=\EE\, vec(X),\\
 vec\big(\PP_E(X)\big)=&(\Theta\T\otimes\Theta\T-I\otimes I)vec(X),
 \end{align*}
with the definition $\EE:=I\otimes E\T+E\T\otimes I$.
We note, that both matrices are nilpotent since they are similar to the strictly lower triangular matrices $I\otimes\tilde E\T+\tilde E\T\otimes I$ and $P\T\otimes P\T-I\otimes I$.
Now, since the matrices $I\otimes E\T$ and $E\T\otimes I$ commute and $\Theta=e^E$, \eqref{expE}, the matrix of the map $\PP_E$ satisfies
\begin{align}\notag
\Theta\T\otimes\Theta\T-I\otimes I
 =&e^{E\T}\otimes e^{E\T}-I\otimes I=e^{E\T\otimes I}e^{I\otimes E\T}-I\otimes I\\\label{Matpp}
 =&e^{\EE}-I =\EE\,\varphi(\EE),
\end{align}
where $\varphi$ is the holomorphic function
\begin{align}\label{varphi}
 \varphi(z)=\frac{e^z-1}{z}=\int_0^1 e^{tz}\,dt,\quad z\in\C.
\end{align}
The function $\varphi$ is a fundamental tool in exponential time integrators, e.g. \cite{HoOs10}.
Translating \eqref{Matpp} back to maps on matrices we write $\varphi(\EE)vec(X)=vec(\Phi_E(X))$.
Some properties of the map $\Phi_E$ are proved in the following lemma directly from an integral representation.
The lemma will use the notion of real holomorphic functions.
This means a function $q(z)$ with real coefficients having an absolutely convergent expansion in a neighborhood of $z=0$.
Since such functions are applied to the nilpotent matrix $E\in\R^{s\times s}$ only, $q(E)$ is always a polynomial of degree not exceeding $s-1$.
\begin{lemma}\label{Lphix}
With the nilpotent matrix $E\in\R^{s\times s}$ let $\Phi_E:\,\R^{s\times s}\to \R^{s\times s}$ be the linear map given by
\begin{align*}
 \Phi_E(X):=\int_0^1 e^{tE\T}Xe^{tE}\,dt.
\end{align*}
Then, with any symmetric matrix $X\in\R^{s\times s}$ the following properties hold.
\\{}
a)\quad $\Phi_E(X)E+E\T\Phi_E(X)=\Theta\T X\Theta-X$, i.e.
\begin{align}\label{LLEPP}
 \PP_E=\LL_E\circ\Phi_E=\Phi_E\circ\LL_E.
\end{align}
b)\quad 
 $\Phi_E:\,\DD^s\to\DD^s,\quad \Phi_E:\,\DD_0^s\to\DD_0^s$.
\\{}
c)\quad 
  With any nonsingular matrix $U$ holds
 $$U\T \Phi_E(X)U=\Phi_{U^{-1}EU}(U\T XU).$$
\qquad
 With any real holomorphic function $q$ satisfying $q(0)=1$ holds
 $$q(E\T)\Phi_E\big(X\big)q(E)=\Phi_E\big(q(E\T) Xq(E)\big).$$
d)\quad For any matrix $K$ from the kernel of $\LL_E$ holds $\Phi_E(K)=K$.
\\{}
e)\quad The kernels of the maps $\LL_E$ and $\PP_E$ are identical.
\end{lemma}
\begin{proof}
a) Partial integration shows that
\begin{align*}
 \Phi_E(X)E=&\int_0^1 e^{tE\T}XEe^{tE}\,dt
  =\Big[e^{tE\T}Xe^{tE}\Big]_0^1-E\T\int_0^1 e^{tE\T}Xe^{tE}\,dt\\
  =&e^{E\T}Xe^{E}-X-E\T \Phi_E(X).
\end{align*}
Reordering terms and recalling $e^E=\Theta$ yields $\Phi_E(X)E+E\T \Phi_E(X)=\LL_E(\Phi_E(X))=\Theta\T X\Theta-X=\PP_E(X)$.
The same computations apply for $\Phi_E(XE+E\T X)$.
\\{}
b) Definiteness follows directly from the integral representation.
In fact, with any $X\in\R^{s\times s}$, $u\in\C^s\setminus\{0\},$ holds
\[ u^\ast \Phi_E(X) u=\int_0^1 (e^{tE}u)^\ast X(e^{tE}u)dt.\]
Hence, the integral is positive or nonnegative whenever $X\in\DD^s$ or $X\in\DD_0^s$, respectively.
\\{}
c) Let $E':=U^{-1}EU$. Then,
\begin{align*}
 U\T\Phi_E( X)U=&\int_0^1 U\T e^{tE\T}U^{-T}U\T XUU^{-1}e^{tE}Udt\\
 =&\int_0^1 e^{t(E')\T}U\T XUe^{tE'}dt.
\end{align*}
A matrix of the form $U=q(E)$ commutes with $E$ leading to $E'=E$.
\\{}
d) Let $K$ be any matrix satisfying $KE=-E\T K$.
Then $KE^j=(-E\T)^jK,\,j\in\N$, and from the series expansion of the exponential function follows $Ke^{tE}=e^{-tE\T}K$.
Hence,
\[ \Phi_E(K)=\int_0^1 e^{tE\T}Ke^{tE}\,dt=\int_0^1 e^{tE\T}e^{-tE\T}K\,dt=K.\]
e) The matrix representation of the linear map $\Phi_{E}$ is $\varphi(\EE)$.
By assertion c) we may consider the map with the transformed matrix $\tilde E\T\otimes I+I\otimes\tilde E\T$ which is strictly lower triangular.
Hence, $\varphi(\tilde E\T\otimes I+I\otimes\tilde E\T)$ is lower triangular with unit diagonal and, hence, nonsingular.
By \eqref{LLEPP} $\Phi_{E}\circ\LL_{E}$ and $\LL_{E}$ have the same kernel.
\end{proof}
\begin{remark}
Since $E$ is nilpotent the map $\Phi_E=\sum_{k=0}^\infty\frac1{(k+1)!} \LL_E^k$ is easily evaluated in practice.
For $X\in\R^{s\times s}$ the recursion $X_0:=X$, $$X_{k}:=\LL_E(X_{k-1})=X_{k-1}E+E\T X_{k-1},\ k=1,2,\ldots$$
terminates with $X_{2s-1}=0$.
This follows inductively from the fact that each application of the transformed map $\LL_{\tilde E}$ introduces one more trivial antidiagonal (southwest to northeast) in $X_k$ starting with $e_1\T X_1e_1=0$.
Hence,
\begin{align}\label{Phireihe}
\Phi_E(X)=\sum_{k=0}^{2s-2}\frac1{(k+1)!}X_k.
\end{align}
\end{remark}
With the definition \eqref{DLLPP} at hand the leading diagonal block $\hat M_{11}$ of the test matrix \eqref{MThtr} may be written in a more compact form.
Its semi-definiteness is necessary for the semi-definiteness of the whole test matrix itself.
With the factorization \eqref{LLEPP} this requires that
\begin{align}\label{dfntlp}
 \hat M_{11}=\LL_E(G\T ZG)-\PP_E(\hat W)=\LL_E\big(G\T ZG-\Phi_E(\hat W)\big)\succeq0.
\end{align}
It is important to note that due to the singularity of the map $\LL_E$ semi-definiteness of an image $\LL_E(X)\succeq0$ is a strong restriction for both $X$ and the image $\LL_E(X)$.
These conditions follow easily from the simple structure of the transformed matrix $\tilde E=DF_0\T$.
Here, the image $U=\LL_{\tilde E}(X)$ is given by
\begin{align}\label{lleij}
  u_{ij}=(j-1)x_{i,j-1}+(i-1)x_{i-1,j},\quad 1\le i,j\le s,
\end{align}
where elements with an index zero are missing.
Hence, the first diagonal element of the image is always zero, $u_{11}=0$, and  the first row and column must vanish completely for $U$ to be semi-definite.
In fact, only a block on the southeast corner of $U$ may be nontrivial.
The precise structure is given in the next lemma.
\begin{lemma}\label{lmspsd}
Let $X\in\R^{s\times s}$ be symmetric and let $U:=\LL_{\tilde E}(X)\succeq0$.
Then, $u_{ij}=0$ for $\min\{i,j\}\le\lfloor(s+1)/2\rfloor$ and also $x_{ij}=0$ for $i+j\le s$.
\end{lemma}
\begin{proof}
By induction over the rows $i=1,\ldots,k:=\lfloor(s+1)/2\rfloor$ we show that $u_{ij}=0,\,j=i,\ldots,s$ and $x_{ij}=0,\,j=i,\ldots,s-i$.
If some diagonal element of $U\in\DD_0^s$ vanishes, $u_{ii}=0$, then by semi-definiteness all elements of row and column number $i$ vanish, as well.
Hence, the first row and column of $U$ are zero, since $u_{11}=0$.
In \eqref{lleij} this gives $0=u_{1j}=(j-1)x_{1,j-1}$, $j=2,\ldots,s$, i.e. the assertion for $i=1$.
Now, by induction with $i\ge2$ we get $u_{ii}=2(i-1)x_{i-1,i}=0$ for $2i-1\le s$, i.e. $i\le k$, and this leads again to $u_{ij}=0$, $j=i,\ldots,s$.
In \eqref{lleij} follows again $0=u_{ij}=(i-1)x_{i-1,j}+(j-1)x_{i,j-1}=(j-1)x_{i,j-1}$ for $j=i,\ldots,s+1-i$ leading to $x_{i\ell}=0$ for $i\le\ell\le s-i$.
\end{proof}
\par
It is also possible to describe the kernel of the map $\LL_{\tilde E}$ on the set of symmetric matrices explicitly.
By \eqref{lleij} and also since $e_s\T\tilde E=0$, the rank-one matrix $e_se_s\T$ is always an element of this kernel.
Other elements can be derived from the proof of the last lemma.
As in the base case of induction the first row of any matrix $K$ satisfying $\LL_E(K)=0$ vanishes with the possible exception of the last element, $k_{1j}=0,\,j=1,\ldots,s-1$.
Now, \eqref{lleij} is a linear one-step recursion along antidiagonals and it follows that also $k_{ij}=0$ for $i+j\le s$.
On antidiagonals with $i+j>s$, $k_{is}\not=0$, this recursion leads to elements with alternating signs.
Hence, corresponding solutions $K$ are symmetric only if the antidiagonal  contains a diagonal element, i.e. if $i+j$ is even.
This result is formulated as a simple lemma which by \eqref{KongL} remains valid with any matrix $E$ similar to $\tilde E$.
\begin{lemma}\label{LKern}
The kernel of the map $\LL_{E}$ on the set of symmetric matrices has dimension $\lfloor(s+1)/2\rfloor$.
\end{lemma}
A discussion is similar for the pre-image of semi-definite matrices $U$.
Since the map $\PP_{\tilde E}$ is sometimes more convenient in the construction of methods (see Section \ref{Secppm}) the following result is formulated for this map.
\example\label{urbu}
For $s=4$ the general symmetric pre-image of a semi-definite matrix $U=\PP_{\tilde E}(X)\in\DD_0^s$ is given by
\begin{align}\label{KSallg}
 X=\begin{pmatrix}
 0&0&0&-\frac34 u_{33}\\
 0&0&\frac14u_{33}&-\frac32x_{33}-\frac34u_{33}+\frac12u_{34}\\
 0&\ast&x_{33}&\frac14u_{33}-\frac13u_{34}+\frac16u_{44}\\
 \ast&\ast&\ast&x_{44}
 \end{pmatrix}.
\end{align}
This means that only $u_{33}, u_{34},u_{44}$ of the image may be nontrivial and $x_{33}$ and $x_{44}$ are parameters of the kernel of $\PP_{\tilde E}$.
\subsection{Explicit feasible solutions}
Returning to the feasibility problem  the test matrix \eqref{MThtr} is rewritten with Lemma~\ref{Lphix} in the form
\begin{align}\notag
 \hat\MM =&\begin{pmatrix}\hat M_{11}&\hat M_{12}\\
   \hat M_{12}\T&\hat W\end{pmatrix}\\\label{MThf}
 =&\begin{pmatrix}
   \LL_E\big(G\T ZG-\Phi_E(\hat W)\big)&\hat W-G\T Z+(G\T ZG)E\\
   \hat W-ZG+E\T(G\T ZG)&\hat W \end{pmatrix}.
\end{align}
By Lemma~\ref{lmspsd} $\hat M_{11}$ has a rank defect $\lfloor(s+1)/2\rfloor$ and by semi-definiteness $\hat M_{12}$ must have the same, at least.
Equation \eqref{MThf} has been used as a definition of the blocks $\hat M_{ij}$ so far.
However, with a matrix $\hat\MM\succeq0$ having an appropriate kernel it may be considered as a set of equations where $\hat\MM$ is the slack variable of the feasibility problem.
Peer methods are suited for stiff problems only if the eigenvalues of the coefficient matrix $\Gmm$ have nonnegative real part.
This property follows directly from \eqref{MThf}.
\begin{lemma}\label{LEWG}
 Let $G\in\R^{s\times s}$ and $Z\in\DD^s$, $\hat W\in\DD_0^s$ (or $\hat W\in\DD^s$) satisfy \eqref{MThf} with $\hat\MM\in\DD_0^{2s}$.
 Then, all eigenvalues of $G$ have nonnegative (positive) real part.
\end{lemma}
\begin{proof}
The equation for the off-diagonal block is written as $G\T Z=(G\T ZG)E+\hat W-\hat M_{12}$.
Adding its transpose leads to
\begin{align*}
 ZG+G\T Z=&(G\T ZG)E+E\T(G\T ZG)+2\hat W-\hat M_{12}-\hat M_{12}\T\\
 =&\Theta\T\hat W\Theta+\hat M_{11}-\hat M_{12}-\hat M_{12}\T+\hat W\\
 \succeq&\Theta\T\hat W\Theta\succeq0.
\end{align*}
The last step used the fact that for  $\hat\MM\in\DD_0^{2s}$ also $\hat M_{11}-\hat M_{12}-\hat M_{12}\T+\hat W\in\DD_0^s$.
A congruence transformation with the inverse factor of a Cholesky decomposition $Z=L_ZL_Z\T$ leads to the symmetric matrix
\[  L_Z\T G L_Z^{-T}+L_Z^{-T}G\T L_Z\succeq L_Z^{-T}\Theta\T\hat W\Theta L_Z^{-1}\succeq0.\]
Hence, the symmetric part of the matrix $L_Z\T G L_Z^{-T}$ is positive semi-definite and all its eigenvalues have nonnegative real part.
Of course, $L_Z\T G L_Z^{-T}$ is similar to $G$.
The argument for $\hat W\in\DD^s$ is the same.
\end{proof}
\par
The leading blocks $\hat M_{11}$ and $\hat M_{12}$ are highly rank deficient by Lemma~\ref{lmspsd}.
Hence, simple solutions may be obtained where these two blocks are chosen zero to start with.
This means that solutions $Z\in\DD^s$, $\hat W\in\DD_0^s$, $G\in\R^{s\times s}$, of the following system are sought:
\begin{align}\label{MSoK}
 \begin{array}{rl}
 \LL_E\big(G\T ZG-\Phi_E(\hat W)\big)=&0,\\
 \hat W-G\T Z+(G\T ZG)E=&0.
 \end{array}
\end{align}
The second equation is not solved for $\hat W$ to eliminate it from the system since symmetry and (semi-) definiteness of $\hat W$ are crucial.
Enforcing these properties would lead to difficult side conditions.
\par
The first equation of the system \eqref{MSoK} contains the unknown matrix  $\Zh=G\T ZG$.
For nonsingular $G$ this matrix is definite again and may be considered to be the unknown weight matrix itself.
Then, the second equation takes the form $0=\hat W-\Zh(G^{-1}-E)$.
The first equation in \eqref{MSoK} means that $\Zh-\Phi_E(\hat W)$ is an element of the kernel of $\LL_E$, for instance zero.
Now, choosing $\hat W\succ0$ and solving the first equation with $\Zh=\Phi_E(\hat W)$ and the second for $H:=G^{-1}=E+\Zh^{-1}\hat W$ yields a special feasible solution.
This gives the following general theorem.
\begin{theorem}\label{Tallgs}
For each $s\in\N$ there exist stiffly accurate A-stable peer methods with $s$ stages and order $s-1$.
\end{theorem}
\begin{proof}
Let $s\in \N$ and let $c_i,\,i=1,\ldots,s$, be an arbitrary set of mutually distinct nodes and $V$ its Vandermonde matrix.
Then, with $E:=VDF_0\T V^{-1}$ and any matrix $\hat W\in\DD^s$ also $\Zh:=\Phi_E(\hat W)\in\DD^s$ by Lemma~\ref{Lphix} and the pair $(\hat W,\Zh)$ solves the first equation in \eqref{MSoK}.
Since the inverse of $\Zh$ exists by definiteness, a solution of the second equation is $H=E+\Zh^{-1}\hat W$.
With these matrices the test matrix \eqref{MThtr} is semi-definite with one nontrivial block $\hat W\in\DD^s$ and the matrix $H$ is nonsingular by Lemma~\ref{LEWG}.
Hence, the matrix $G=H^{-1}$ and $B=(I-GE)\Theta$ from \eqref{ordsme} define an A-stable peer method of order $s-1$ according to Theorem~\ref{TAAStb} and \eqref{MThtr}.
\end{proof}
\par
Unfortunately, Theorem~\ref{Tallgs} has limited practical relevance since the coefficient matrix $G$ will be a full matrix, in general, leading to prohibitively expensive peer methods.
Turning attention to diagonally implicit or parallel methods the discussion is restricted to the simpler case of a nonsingular coefficient matrix $G$.
\subsection{Methods with nonsingular coefficient matrix $G$}\label{SGreg}
For nonsingular $G$ it is convenient to rewrite the test matrix \eqref{MThf} in terms of $\Zh=G\T ZG\in\DD^s$ and $H=G^{-1}$.
Now, the test matrix has a generic form depending on the quadruple $(E,H,\hat W,\Zh)$.
This form remains unchanged under congruence-similarity transformations and is denoted by
\begin{align}\label{MMGen}
 \MG(E,H,\hat W,\Zh):=
 \begin{pmatrix}\LL_E(\Zh)-\PP_E(\hat W)&\hat W-\Zh(H-E)\\
 \hat W-(H-E)\T\Zh&\hat W \end{pmatrix}.
\end{align}
\begin{remark}
Congruence-similarity transformations with block diagonal matrices $\diag(U,U)$ establish an equivalence relation between test matrices of the form \eqref{MMGen}.
In fact by \eqref{KongL}, \eqref{KongP} it follows that
\begin{align}\label{equrel}
  \MG(E,H,\hat W,\hat Z)\succeq0\ \iff
 \ \MG(U^{-1}EU,U^{-1}HU,U\T\hat WU,U\T\hat ZU)\succeq0
\end{align}
for nonsingular $U\in\R^{s\times s}$.
This property will be helpful in the construction of efficient peer methods where $H=G^{-1}$ has lower triangular or diagonal form.
However, the resulting methods only have high order if their $E$-matrix is Vandermonde-similar to $\tilde E=DF_0\T$.
This distinguished Nordsieck version using $U=V$ is used frequently and the actual value of $\MG$ is denoted as $\tilde\MM=\MG(DF_0\T,\tilde H,\tilde W,\tilde Z)$.
In fact, the rank defect of the first block has been identified by Lemma~\ref{lmspsd} for $\tilde M_{11}$.
By semi-definiteness the same rows as in $\tilde M_{11}$ must also vanish in $\tilde M_{12}$ leading to an explicit decomposition in block diagonal form $\tilde\MM=0_k\oplus\MM_D$ with $k=\lfloor(s+1)/2\rfloor$ mentioned in the introduction.
It will be seen in some examples that the nontrivial block may be definite, $\MM_D\in\DD^{2s-k}$, which can be checked reliably with floating point arithmetic.
The construction of A-stable methods will make use of this explicit block decomposition.
\end{remark}
With \eqref{MMGen} the original feasibility problem is again written in the form
\begin{align}\label{TMhat}
 \MG(E,H,\hat W,\Zh)=&
\begin{pmatrix}\hat M_{11}&\hat M_{12}\\
   \hat M_{12}\T&\hat W\end{pmatrix}=\hat\MM
\end{align}
with the matrix $\hat\MM$ as slack variable.
With the results of the last subsections it is possible to derive a rather explicit parametrization of solutions of the feasibility problem $\MG(E,H,\hat W,\Zh)\succeq0$.
We note that in the Nordsieck form pre-images $X$ with $\PP_{\tilde E}(X)=\tilde M_{11}\in\DD_0^s$ can be computed easily, see Example~\ref{urbu}.
\begin{lemma}\label{LLparam}
Let $E\in\R^{s\times s}$ be nilpotent and $\hat W_0\in\DD^s$.
With a matrix $K=K\T$ from the kernel, $\LL_E(K)=0$, and $N=N\T\in\R^{s\times s}$ such that $\PP_E(N)=\hat M_{11}\in\LL_E(\DD_0^s)\cap\DD_0^s$ let $\hat W:=\hat W_0-K-N\in\DD^s$.
Then, for any matrix $\hat M_{12}\in\R^{s\times s}$ such that $\hat\MM\in\DD_0^{2s}$ the triple of matrices
\begin{align}\label{exlsgv}
 \Zh =\Phi_E(\hat W_0),\quad 
 \hat W=\hat W_0-K-N,\quad
  H=E+\Zh^{-1}(\hat W+\hat M_{12}),
\end{align} 
is a feasible solution $\MG(E,H,\hat W,\Zh)\succeq0$ with $\Zh\succ0$.
\end{lemma}
\begin{proof}
By Lemma~\ref{Lphix} the equation for the first block may be replaced by $\LL_E(\hat Z)-\PP_E(\hat W)-\hat M_{11}=\LL_E\big(\hat Z-\Phi_E(\hat W+N)\big)=0$ and its general solution is $\hat Z=\Phi_E(\hat W+N)+K=\Phi_E(\hat W+N+K)$ with $\LL_E(K)=0$.
Assuming $\hat W_0:=\hat W+N+K\in\DD^s$ gives $\Zh\in\DD^s$ by part b) of Lemma~\ref{Lphix}.
Now, since $\Zh$ is nonsingular, the off-diagonal block equation may be solved uniquely for $H$.
\end{proof}
\\
Before turning to the construction of practical methods with restrictions on the shape of $H=G^{-1}$ two interesting cases of transformations \eqref{equrel} of known solutions are considered.
First, part c) of Lemma~\ref{Lphix} shows that transformation matrices commuting with $E$ lead to simple relations.
In fact, congruences with a real function $q(E)$ leave the kernel invariant:
if $\LL_E(K)=0$, then
\begin{align*}
\LL_E\big(q(E\T)Kq(E)\big)=&q(E\T)Kq(E)E+E\T q(E\T)Kq(E)\\
 =&q(E\T)(KE+E\T K)q(E)=0.
\end{align*}
Hence, any solution \eqref{exlsgv} without slack elements ($\hat M_{11}=0,\,\hat M_{12}=0$) yields a new solution by the similarity transformation
\begin{align}\label{Lsgqe}
 q(E)^{-1}Hq(E)=E+\big(\Phi_E\big(q(E\T)\hat W_0q(E)\big)\big)^{-1}
  \big(q(E\T)\hat W_0q(E)+K\big).
\end{align}
Hence, if $E$ is already Vandermonde-similar to $\tilde E$ by $E=VDF_0\T V^{-1}$, \eqref{Lsgqe} may be used as a post-processing to change the shape of $H$.
Effects of this transformation on slack terms have to be discussed on a case-by-case basis.
\par
Of fundamental interest may be an exceptional transformation eliminating the weight matrix $\Zh$.
It yields a simple, distinguished representative of the whole equivalence class of relation \eqref{equrel}.
It uses the Cholesky decomposition $\Zh=L_ZL_Z\T$ and introduces the transformed matrices $\check W:=L_Z^{-1}\hat WL_Z^{-T}$, $\check E:=L_Z\T EL_Z^{-T}$, $\check H:=L_Z\T HL_Z^{-T}$.
Then, the matrix $\hat\MM$ from \eqref{TMhat} is congruent with
\begin{align}\label{Mhaken}
  \MG(\check E,\check H,I,\check W)   
  =&\begin{pmatrix}
  \LL_{\check E}(I)-\PP_{\check E}(\check W)
   &\check W-\check H+\check E\\
   \check W-\check H\T+\check E\T &\check W \end{pmatrix}=:\check\MM,
\end{align}
depending linearly on $\check W$ and $\check H$.
Since $Z$ has been eliminated an explicit solution representation for $\check W$ is required now.
This is obtained with the inverse map $\Psi_E$ of $\Phi_E$ satisfying $\Phi_E(\Psi_E(X))=X$ for $X\in\DD^s$ and nilpotent $E$.
Similar to \eqref{Matpp} its matrix representation is given by $\psi(\EE)vec(X)=vec(\Psi_E(X))$ where $\psi$ is the reciprocal of the function $\varphi$ from \eqref{varphi},
\begin{align}\notag
 \psi(z):=\frac1{\varphi(z)}=&-\frac{z}2+\frac{z}2\Big(\frac{e^z+1}{e^z-1}\Big)
 \\\label{defpsi}
 =&-\frac{z}2+1-\sum_{k=1}^\infty(-1)^k\frac{\beta_k}{(2k)!}z^{2k}.
\end{align}
The series expansion of $\psi$ has only one odd term $-z/2$ and the even terms define the Bernoulli numbers $\beta_k$, and the radius of convergence is $\pi$. 
In fact the expansion corresponds to the Magnus series $\psi=d\exp^{-1}$, see \cite[\S IV.7]{HLW07}.
Using $\Psi_{\check E}(I)$ for representing the general feasible solution in \eqref{Mhaken} without slack, i.e. $\check M_{11}=\check M_{12}=0$, shows that the skew-symmetric part of the matrix $\check H$ is identical with that of $\check E$.
\begin{lemma}\label{Lcheck}
Let $\check E\in\R^{s\times s}$ be nilpotent and $\check W=\Psi_{\check E}(I)-\check K\in\DD^s$ with $\check K=\check K\T$ and $\LL_{\check E}(\check K)=0$.
Then, a solution with $\check M_{11}=\check M_{12}=0$ in \eqref{Mhaken} is given by $\check W$ and
$\check H=\check E+\check W=\check H_{sym}+\check H_{skew},$
where
\begin{align}\label{Hhaken}
 \check H_{sym}=&\frac12(\check E+\check E\T)+\Psi_{\check E}(I)+\check K
 =I-\sum_{k=1}^\infty(-1)^k\frac{\beta_k}{(2k)!}\LL_{\check E}^{2k}(I)+\check K\\\notag
 \check H_{skew}=&\check E_{skew}=\frac12(\check E-\check E\T).
\end{align}
\end{lemma}
\begin{proof}
By Lemma~\ref{Lphix} holds $\LL_{\check E}(I)=\LL_{\check E}\big(\Phi_{\check E}(\Psi_{\check E}(I))\big)=\PP_{\check E}\big(\Psi_{\check E}(I)\big)$ and the general solution of $\LL_{\check E}(I)-\PP_{\check E}(\check W)=\PP_{\check E}\big(\Psi_{\check E}(I)-\check W\big)=0$ is $\check W=\Psi_{\check E}(I)+\check K$ with a symmetric matrix $\check K$ from the kernel $\LL_{\check E}(\check K)=\PP_{\check E}(\check K)=0$.
Solving $\check M_{12}=0$ by $\check H=\check E+\Psi_{\check E}(I)$ the linear term from the expansion \eqref{defpsi} combines with the term $\check E$ and gives the skew symmetric matrix $E-\frac12\LL_{\check E}(I)=\frac12(E-E\T)$.
Since the remaining terms are symmetric the assertion follows.
\end{proof}
\\
The lemma shows that introducing the map $\Psi_E$ provides additional insight in the structure of solutions.
However, since $\Psi_E$ does not preserve definiteness the assumption $\check W\in\DD^s$ had to be added explicitly.
Using the definiteness preserving map $\Phi_E$ was more convenient for the theoretical result of Theorem~\ref{Tallgs}.
\par
After choosing the matrix $\check E$ Lemma~\ref{Lcheck} provides a simple feasible solution which may be complemented by slack terms $\check M_{1j}\not=0$, $j=1,2$, having the proper rank.
From it the whole equivalence class of A-stable methods can be obtained by similarity \eqref{equrel}.
However, it is still necessary to identify methods of higher order having a Vandermonde structure with $E=VDF_0\T V^{-1}$ and being of practical interest with a triangular or diagonal matrix $H=G^{-1}$.
For these two cases different approaches are used.
\section{Diagonally implicit peer methods}\label{Secspm}
The methods from Theorem~\ref{Tallgs} are not suitable for practical computations since $G\in\R^{s\times s}$ is a full matrix, in general.
In this case time steps of \eqref{PeerV} are too expensive requiring the solution of a nonlinear system of $s\cdot n$ simultaneous equations.
For methods with a lower triangular coefficient matrix $G$ the $s$ stage systems are decoupled.
Such peer methods have been discussed in \cite{BWPS10}, \cite{SWB12}.
Singly-implicit methods with constant main diagonal $g_{ii}=\gamma,\,i=1,\ldots,s$, have the additional benefit that costly matrix decompositions may be computed only once per time step.
\par
Similarity transformations with lower triangular matrices do not destroy the lower triangular form of $H$.
A first transformation uses the L-factor of an explicit LU decomposition $V=L_VU_V$ of the Vandermonde matrix $V$ \cite{Turn66} given by
\begin{align}\label{VdMU}
 U_V^{-1}=&\begin{pmatrix}
  1&-c_1&c_1c_2&-c_1c_1c_2&\ldots\\
 &1&-(c_1+c_2)&c_1c_2+c_1c_3+c_2c_3&\ldots\\
 &&1&-(c_1+c_2+c_3)&\ldots\\
  &&&\ddots\end{pmatrix}\\\label{VdML}
 (L_V^{-1})_{ij}=&\prod_{k=1\atop k\not=j}^i\frac1{c_j-c_k},\ 1\le j\le i\le s.
\end{align}
In fact, this is the matrix version of Newton interpolation with the nodes $(c_i)$ since $L^{-1}$ contains the factors for the divided differences and $U^{-1}$ contains the elementary symmetric functions.
The factor $L_V$ does not have unit diagonals, all its entries depend on the differences between nodes only.
\par
It is possible to obtain a final solution $H$ with lower triangular form obeying \eqref{exlsgv} with $E=V\tilde EV^{-1}$ from a simpler, equivalent version.
In fact, a similarity transformation with the factor $L_V$ yields
\begin{align}\label{LHL}
 L_V^{-1}HL_V 
  =&L_V^{-1}EL_V+\big(\Phi_{L_V^{-1}EL_V}(L_V\T\hat W_0 L_V)\big)^{-1}L_V\T(\hat W+\hat M_{12})L_V,
\end{align}
where $L_V^{-1}HL_V$ also has lower triangular form while $L_V^{-1}EL_V=U_VDF_0\T U_V^{-1}=U_V\tilde EU_V^{-1}$ is strictly upper triangular.
In fact, its entries depend only on differences of the nodes and for $s\le4$ this matrix is explicitly given by the leading submatrices of
\begin{align}\label{LEL}
L_V^{-1}EL_V=U_VDF_0\T U_V^{-1}
  =\begin{pmatrix} 0&1&c_1-c_2&(c_1-c_2)(c_1-c_3)\\
 &0&2&c_1+c_2-2c_3\\ &&0&3\\&&&0  \end{pmatrix}.
\end{align}
We note that the terms $K+N$ in the matrix $\hat W$, see \eqref{exlsgv}, contribute to \eqref{LHL} a matrix $L_V\T(K+N)L_V=R_V^{-T}(\tilde K+\tilde N)R_V^{-1}$ where $X=\tilde K+\tilde N$ is given explicitly by \eqref{KSallg} for $s=4$.
Since upper triangular congruence transformations do not change the zero pattern in \eqref{KSallg} the kernel and slack terms $L_V\T(K+N)L_V$ in \eqref{LHL} only contribute to the southeast antidiagonals.
\par
It is clear that one can go back from a lower triangular solution $L_V^{-1}HL_V$ of the form \eqref{LHL} with a suitable strictly upper triangular matrix \eqref{LEL} to a solution \eqref{exlsgv} by the similarity transformation with the factor $L_V$ of the Vandermonde matrix.
The open question remaining is to find solutions such that the $E$-matrix has the form from \eqref{LEL}, $U_VDF_0\T U_V^{-1}=U_V\tilde EU_V^{-1}$.
Here, the following lemma will be applied to the $Z$-free versions \eqref{Mhaken}, \eqref{Hhaken}.
It is proved only for the non-defective case.
\begin{lemma}\label{Ldrnf}
Every non-defective matrix $A\in\C^{s\times s}$ possesses a {\em triangular canonical form}
\begin{align}\label{drnf}
  A=U_AL_AU_A^{-1}
\end{align}
where $L_A$ is lower triangular and $U_A$ is upper triangular and nonsingular.
\end{lemma}
\begin{proof}
Let $A=XJX^{-1}$ be the Jordan canonical form of $A$ where $J$ is diagonal by assumption.
Then, with some permutation matrix $\Pi$ there exists the LU-decomposition $X^{-1}=\Pi L_XU_X$ and the factorization $A=U_X^{-1}\big(L_X^{-1}\Pi\T J\Pi L_X)U_X$ proves the assertion with  $L_A=L_X^{-1}\Pi\T J\Pi L_X$ and $U_A=U_X^{-1}$.
\end{proof}
\par
The canonical form may not exist in some exceptional cases like a Jordan block $\begin{pmatrix} a&b\\0&a \end{pmatrix},\ b\not=0,$
in dimension 2.
In fact, the lemma is related to the canonical form computed by the classical LR-iteration of Rutishauser \cite{Wilk65}.
However, the existence of this decomposition can be used as a criterion in the search for methods.
Obviously, the main diagonal of $L_A$ contains the eigenvalues of $A$.
The diagonal elements of $U_A$ may be chosen arbitrarily but non-zero  adding some degrees of freedom in the forthcoming application.
This triangular canonical form will be used to construct or reconstruct A-stable peer methods from a $Z$-free representation  (see Lemma~\ref{Lcheck}) and the knowledge of the matrix $\check E$ and nontrivial kernel and slack terms alone.
We consider this approach primarily as a theoretical tool using formula manipulation or exact arithmetic since numerical computations may be too inaccurate for multiple eigenvalues.
The input of the following algorithm is the strictly upper triangular matrix $\check E$ and kernel and slack elements from $\check M_{11}$ which are combined in a matrix $\check X$. 
Of course, the choice of such terms is restricted by the definiteness of $\check W=\Psi_{\check E}(I)-\check X$.
The image $\PP_{\check E}(\check X)$ is required to be semidefinite and has the zero structure of Lemma~\ref{lmspsd}.
The two nodes  $c_1,c_s$ may be chosen freely, where the choice $c_s=1$ is convenient in practice.
\par\noindent{\bf Algorithm.}\\
{\em Input:}
Strictly upper triangular matrix $\check E\in\R^{s\times s}$, $\check e_{i,i+1}\not=0$, $i=1,\ldots,s-1$ and $\check X=\check X\T$ such that $\PP_{\check E}(\check X)\in\DD_0^s$.\\
{\em Steps:}
\begin{enumerate}
\item
Compute $\check W=\Psi_{\check E}(I)-\check X$ and $\check H=\check E+\check W$, check definiteness $\check W\in\DD^s$.
\item
Compute the triangular canonical form \eqref{drnf} for $\check H=U_HL_HU_H^{-1}$ with the diagonals of $U_H$ chosen as $u_{ii}=(i-1)!/\prod_{j=2}^i\check e_{j-1,j},$ starting with $u_{11}=1$.
\item
Apply the congruence-similarity transformation with $U_H$ to \eqref{Hhaken} by computing $E'=U_H^{-1}\check E U_H$, $W'=U_H\T\check W U_H$, $Z'=U_H\T U_H$ satisfying $L_H=E'+(Z')^{-1}W'$.
The matrix $H'=L_H$ is known from step (2).
\item
Identify $E'$ with \eqref{LEL}, i.e. solve the equation $E'=U_V DF_0\T U_V^{-1}$ for the differences of the nodes $c_i-c_1$, $i=2,\ldots,s-1$.
In the first superdiagonal the identity $e_{i,i+1}'=i,\,1\le i<s$ is satisfied by the choice of the diagonals of $U_H$ in step (2).
\item 
Apply a final congruence-similarity transformation with the Vandermonde L-factor $L_V$  from \eqref{VdML} to the matrices in order to obtain the coefficients of the A-stable method
\[ H=L_VH'L_V^{-1},\  E=L_V E'L_V^{-1}=V\tilde EV^{-1},\ \]
and the weight matrices $\hat W=L_V^{-T}W'L_V^{-1}$, $\Zh=L_V^{-T}Z'L_V^{-1}$.
\end{enumerate}
The procedure is based on the following congruences of the equation $\check H=\check E+\check W$:
\[ \begin{array}{rc@{+}c}
  \check H=U_HL_HU_H^{-1}=&\check E&\check W\\
  (U_H\T U_H)\quad L_H\ =&(U_H\T U_H)\underbrace{U_H^{-1}\check EU_H}_{\|}&(U_H\T \check WU_H)\\
  Z'\quad\qquad H'\quad=&\quad Z'\quad U_V\tilde EU_V^{-1}&W'\\
  (L_V^{-T}Z'L_V^{-1})(L_V H'L_V^{-1})=&(L_V^{-T}Z'L_V^{-1})(L_V U_V\tilde EU_V^{-1}L_V^{-1})&(L_V^{-T}W'L_V^{-1})
 \end{array}
\]
\begin{remark}
There are two critical points in this algorithm.
For methods of practical interest eigenvalues have to be real.
Complex eigenvalues can easily be avoided in Step~2 with the criterion of Sylvester \cite[Th.1.4.4]{Praso} by computing traces of powers of $\check H$ and a LU decomposition of a certain Hankel matrix.
More crucial is Step~4 since \eqref{LEL} has only $s-2$ degrees of freedom for equating the ${s-1\choose 2}$ elements $e_{ij}',\ i+2 \le j\le s$.
Hence, for $s\ge4$ the algorithm will only be successful if the input $\check E$ satisfies certain unknown side conditions.
However, the algorithm may be used to decode the full triple $(H,\hat W,\Zh)$ plus nodes of a known A-stable method from a compact representation consisting of the triangular matrix $\check E$ and optional kernel or slack contributions.
\end{remark}
Only for three-stage methods Step~4 consists of one single equation which can always be solved by  $c_2=c_1-e_{13}'$.
This is demonstrated now by constructing an A-stable singly-implicit three-stage method.
\example\label{bspsq3}
For $s=3$ the algorithm is used with $\check E$ and one element $\check\kappa e_3e_3\T$, $\check\kappa>0$, from the kernel of $\LL_{\check E}$. 
This ansatz gives $\check H=\frac12(\check E-\check E\T)+I+\frac1{12}\LL_{\check E}^2(I)-\frac1{720}\LL_{\check E}^4(I)+\check\kappa e_3e_3\T$.
The conditions for a threefold eigenvalue $\eta$ of $\check H$ can be solved for the parameters $\check\kappa,\check e_{13}$ and $\check e_{23}$ with expressions containing square roots.
Enforcing positive radicands and positive diagonals of the LU-decomposition of $\check W$ leaves nontrivial admissible regions in the $(\eta,\check e_{12})$-plane.
Looking for parameters with small radicands leads to the choice $\eta=5/2$, $\check e_{12}=3/2$ yielding $\check\kappa=174/175$, $\check e_{13}=-\frac3{14}\sqrt{91}$, $\check e_{23}=\frac{24}{35}\sqrt{35}$, and
\begin{align*}
 \check H=\begin{pmatrix}
  1&\frac34&\frac{-15\sqrt{13}+12\sqrt5}{140}\sqrt7\\
  -\frac34&\frac{11}8&\frac{96\sqrt5-15\sqrt{13}}{280}\sqrt7\\
  \frac{15\sqrt{13}+12\sqrt5}{140}\sqrt7&
  \frac{-96\sqrt5-15\sqrt{13}}{280}\sqrt7&\frac{41}8
\end{pmatrix}.
\end{align*}
The factors of the triangular canonical form \eqref{drnf} of $\check H$ are
\begin{align*}
 L_H=\begin{pmatrix}
  \frac52&\\
 \frac{-81-9\sqrt{65}}{40} &\frac52&\\
  \frac{27(4+\sqrt{65})}{70}&\frac{9(\sqrt{65}-17)}{28}&\frac52
 \end{pmatrix},\ 
 U_H=\begin{pmatrix}
 1&\frac{11-\sqrt{65}}{6}&\frac{11-\sqrt{65}}{36} \\ &\frac23&\frac7{18}\\&&\frac{\sqrt{35}}{18}
 \end{pmatrix}.
\end{align*}
The diagonal elements of $U_H$ are chosen to give $e'_{i,i+1}=i,i=1,2,$ for $E'=U_H^{-1}\check EU_H$ and equating the only general element $e_{13}'\stackrel!=c_1-c_2$ from \eqref{LEL} gives $c_2=c_1+(259-21\sqrt{65})/84$.
Fixing the free nodes as $c_1=-1$, $c_3=1$ the final transformation yields  $H=L_V H'L_V^{-1}$ and the methods parameter matrix is obtained as its inverse
\[ G=\begin{pmatrix}
 \frac25&0&0\\
 \frac{207+15\sqrt{65}}{500}&\frac25&0\\
 \frac{-3672+1800\sqrt{65}}{30625}&
 \frac{-468+180\sqrt{65}}{1225}&\frac25
\end{pmatrix}.\]
The transformed weight matrix $\Zh=L_V\T Z'L_V=(U_H^{-1}L_V)\T U_H^{-1}L_V$ is given by
\[
 \Zh=\begin{pmatrix}
 \frac{129447+7757\sqrt{65}}{345744}&\frac{124345-9465\sqrt{64}}{749112}
 &\frac{-3315-41\sqrt{65}}{91728}\\
 \ast&\frac{760215-18615\sqrt{65}}{1997632}&
 \frac{-3665+561\sqrt{65}}{122304}\\
 \ast&\ast&\frac{339+29\sqrt{65}}{7488}
 \end{pmatrix}
\]
and $\hat W=L_V\T W'L_V=(U_H^{-1}L_V)\T \check W U_H^{-1}L_V$ can be computed by $\hat W=\Psi_E(\Zh)+\kappa \ell\ell\T$ where $\kappa=(841+87\sqrt{65})/49920$ and $\ell\T=((13+9\sqrt{65})/98,-(111+9\sqrt{65})/98,1)$ is a multiple of the last row of $L_V$.
With these data semi-definiteness of the test matrix $\hat\MM$ \eqref{TMhat} can be checked exactly.
Actually, the test matrix is zero with the exception of the last block $\hat\MM_{22}=\hat W$ by construction.
\par
It is difficult to apply this construction with $s\ge4$.
First, the occurring algebraic equations can no longer be solved with simple square roots.
Secondly, the matrix $E'$ contains 3 general elements while \eqref{LEL} has only 2 degrees of freedom.
In order to show that singly-implicit 4-stage A-stable peer methods exist we give an example computed by some numerical iteration in the Nordsieck form.
It uses a slack matrix $\tilde M_{11}=\diag(0,0,1,2)$ of rank two and kernel $4e_4e_4\T$.
In this form numerical verification of A-stability is reliable since the nontrivial part of the test matrix $\tilde M$ is definite.
In order to save space the $Z$-free formulation of this method is given.
\example\label{bspsq4}
An A-stable singly-implicit 4-stage peer method may be reconstructed with the algorithm from its compact form $\check W=\Psi_{\check E}(I)+ X$ with
\begin{align*} 
 \check E\doteq\begin{pmatrix}
  0&1.828746674&-2.100327482&-0.0374802335\\
   &0&1.159613679&0.0397218240\\
   &&0&1.105306335\\&&&0
 \end{pmatrix}.
\end{align*}
In this form the nontrivial part of the slack matrix $\PP_{\check E}(X)$ is
\[ \begin{pmatrix}
  1.124278133&-0.2749000029\\
 -0.2749000029&0.3724460117
\end{pmatrix}
\]
and the kernel is determined by $x_{33}=-0.845690710$, $x_{44}=-0.7361686650$.
As mentioned above the fourfold eigenvalue $\eta$ of $\check H$ and $H$ can not be computed with sufficient accuracy, its accurate value is $\eta\doteq1.80350113085004$, and the nodes are $(c_i)=(-0.889874593986289,0.522100340305431, -0.297184898847891, 1)$.
We do not go into more detail here since this method may not be of practical interest because efficient methods should satisfy several criteria like those listed in \cite{SWB12}.
The construction of efficient A-stable peer methods will be the topic of a forthcoming paper.
\par
In principle, the feasibility problem for diagonally implicit methods contains so many parameters that one may expect that the algebraic criterion for A-stability can be satisfied for larger stage numbers $s$ as well.
\section{Parallel peer methods}\label{Secppm}
Due to an explicit representation of the matrix $\tilde G$ the design of parallel methods can be completely discussed in the Nordsieck version $\MG(DF_0\T,\tilde H,\tilde W,\tilde Z)$.
Although A-stability is discussed only for constant stepsizes here, peer methods should perform well on general time grids.
Then, zero stability of stiffly-accurate parallel methods requires that the diagonal matrix $G$ has non-constant diagonal elements and does not commute with $V$, see \cite{MPPW}.
Using the Nordsieck version it is necessary to identify matrices $\tilde G=V^{-1}GV$ which are similar to a diagonal matrix $G$ by a Vandermonde transformation $V$. 
The following necessary condition is related to a representation from \cite{SWE05}.
It uses the standard notation of square brackets for the commutator of matrices.
\begin{lemma}\label{SGtilde}
Let $V\in\R^{s\times s}$ be a Vandermonde matrix with distinct nodes $c_1,\ldots,c_s$ and let $p(x)=(x-c_1)\cdots(x-c_s)=x^s+\sum_{i=1}^s p_ix^{i-1}$ be the node polynomial.
The vector ${\bf p}=(p_i)\in\R^s$ contains its coefficients.
If $G\in\R^{s\times s}$ is diagonal, then $\tilde G=V^{-1}GV$ satisfies
\begin{align}\label{commutgt}
  \big[F_0-{\bf p}e_s\T,\tilde G\big]=0.
\end{align}
From \eqref{commutgt} follows the rank condition
\begin{align}\label{GFRang}
 rank\begin{pmatrix}
  [F_0,\tilde G]\\ e_s\T\tilde G
 \end{pmatrix}_{\ast,1\ldots s-1}=1
\end{align}
on the submatrix of dimension $(s+1)\times(s-1)$ indicated by the subscripts.
\end{lemma}
\begin{proof}
The diagonal elements of $G$ may be considered as function values of some polynomial $g$ of degree $s-1$ at the nodes, i.e. $g_{ii}=g(c_i),\,i=1,\ldots,s$.
Introducing the matrix $C=\diag(c_i)$ this means that $G=g(C)$.
Since the nodes are the zeros of $p$ it is well known that $CV=VF$ holds with the Frobenius companion matrix $F:=F_0-{\bf p}e_s\T$.
Hence, $V^{-1}GV=V^{-1}g(C)V=g(F)$ is a function of $F$ and commutes with it.
Writing \eqref{commutgt} as $[F_0,\tilde G]={\bf p}e_s\T\tilde G-\tilde G{\bf p}e_s\T$ shows that the rows of $[F_0,\tilde G]$ are multiples of the last row $\tilde G$ itself with the exception of the last column.
\end{proof}
\par
With given node polynomial the condition \eqref{commutgt} is a simple linear system.
The second version \eqref{GFRang}, however, allows to characterize candidates for transformed diagonal coefficient matrices $\tilde G$ without knowledge of the node polynomial.
Rank 1 can be enforced by a set of $2\times2$-determinants leading to quadratic restrictions on $\tilde G$.
The coefficients of the polynomial can be recovered from the equations $([F_0,\tilde G]-{\bf p}e_s\T\tilde G)_{1\ldots s-1}=0$.
Actually, if $c_s=1$ is one of the nodes which is convenient in practice, then $\ones\T{\bf p}+1=0$ and the matrix from \eqref{GFRang} has trivial column sums, $(\ones\T[F_0,\tilde G]+e_s\T\tilde G)_{1\ldots s-1}=0$.
\example\label{bsppl3}
Construction of a parallel 3-stage peer method.
The essential parameter of the construction is the matrix $\tilde W\in\DD^3$ and the ansatz
\[ \tilde W_0=q(\tilde E)\T D_Wq(\tilde E),
 \ \tilde W=\tilde W_0+K,
 \ K=\begin{pmatrix}&&-2\kappa_1\\ &\kappa_1&0\\ -2\kappa_1&0&\kappa_3 \end{pmatrix},
\]
$D_W=\diag(1,\frac16d_{22},\frac1{18}d_{33})$, and $q(z)=1+q_1z+q_2z^2$ is used, i.e. $\tilde\MM_{11}=0$.
It takes advantage of property \eqref{Lsgqe} and leads to simpler equations in the side conditions than a general $\tilde W_0\in\DD^s$.
Condition \eqref{GFRang} is enforced for the inverse $\tilde H=\tilde G^{-1}$ by the 3 determinants
\[ \delta_i:=\left|\begin{matrix} [F_0,\tilde H]_{i1}&[F_0,\tilde H]_{i2}\\
 \tilde h_{31}&\tilde h_{32}\end{matrix}\right|,\ i=1,2,3.\]
The choice $\kappa_1=1/3$ simplifies the conditions and $d_{22}:=6q_1^2-12q_2-1$ cancels both $\delta_1$ and $\delta_2$.
Solving $\delta_3\stackrel!=0$ for $\kappa_3$ and annihilating the sum in the first column $(\ones\T[F_0,\tilde H]+e_s\T\tilde H)_1=0$ with $d_{33}$ leaves some region in the $(q_1,q_2)$-plane where $\tilde Z=\Phi_{\tilde E}(\tilde W_0)$ and $\tilde W$ are definite.
Finally, looking for rational solutions with small denominators suggests the choice $q_1=3$ and $q_2=17/4$.
The resulting nodes are $(c_i)=(0,2,1)$ and transforming back the coefficient matrix is obtained as $G=\diag\big(\frac25,\frac{20}{29},\frac5{11})$.
The weight matrices in the original form \eqref{TMhat} are
\begin{align*}
 \hat W=\frac1{12}\begin{pmatrix}
 23&20&-50\\ 20&37&-52\\ -50&-52&116
 \end{pmatrix},\quad
 \hat Z=\begin{pmatrix}
  1&\frac53&-\frac52\\ \frac53&5&-5\\-\frac52&-5&\frac{20}3 \end{pmatrix}.
\end{align*}
\par
So far, the system \eqref{TMhat} has been discussed as a nonlinear problem for $H$.
However, if a peer method of order $s-1$ is given by fixing the nodes and the matrix $G$ or $H=G^{-1}$, \eqref{TMhat} is a linear system for the unknown weight matrices $\hat Z$ and $\hat W$ where only positive definite solutions are of interest.
This corresponds to the original approach to the feasibility problem from the literature, e.g. \cite{HeHi09}.
\par
A simple class of implicit parallel methods has been discussed in \cite{MPPW} with the focus on zero stability for non-uniform grids.
With the diagonal matrix $G$ of the special form $G=g_0I+g_1 C$, $g_0,g_1\in\R_+$, $C=\diag(c_i)$, the transformed matrix $\tilde B$ is upper triangular leading to a simple proof for uniform zero stability if the parameter $g_1$ lies in a certain interval containing the choice $g_1=2/5$ for $s=4$.
We now verify A-stability of a 4-stage method of this form with prescribed nodes.
\example\label{bsppl4} 
Construction of a parallel 4-stage method from the class discussed in \cite{MPPW}.
In order to obtain rational solutions the parameters from \cite{MPPW} are approximated by $g_1=2/5$ and nodes $c=(-1,-2/5,2/5,1)$.
Only $g_0$ is considered a variable in the design.
With $\tilde Z=(\tilde z_{ij})\in\DD^4$, $\tilde z_{11}=1$, the general ansatz $\tilde W=\Psi_{\tilde E}(\tilde Z)+X$ is used with $X=K+N$ from \eqref{KSallg}. 
Lemma~\ref{LKern} shows that for $s=4$ the first two rows and columns of $\tilde\MM_{11}$ vanish.
This restriction is obeyed already by using the matrix $X$ in the form \eqref{KSallg}.
However, semidefiniteness also requires that the first two rows of $\tilde \MM_{12}$ be zero.
This gives a linear system with 8 equations which can be solved for the parameters $\tilde z_{12},\tilde z_{13},\ldots,\tilde z_{24}$ and $u_{33},u_{34}$ from \eqref{KSallg}.
Due to the upper triangular form of $\tilde E$ the number of parameters in the entries of $\tilde\MM=(\tilde m_{ij})$ increases from top-left to bottom-right.
For this reason the parameter $u_{33}$ is required since it is the only free parameter in the equation $\tilde m_{38}=0$ besides $g_0$.
These dependencies also facilitate the choice of the remaining parameters by enforcing non-negativity of the minors of $\tilde \MM$ step-wise while keeping a nontrivial admissible interval for the parameter $g_0$.
At the end it was found that the LU-decomposition of $\tilde\MM$ indeed has 6 positive diagonal elements for $g_0\in[0.789,0.822]$.
Hence, with $g_0=4/5$ the method is proven A-stable and its coefficient matrix is $G=\diag(2/5,16/25,24/25,6/5)$.
Since the weight matrices have smaller numbers in the Nordsieck form we show this version.
\begin{align*}
\tilde Z=\begin{pmatrix}
 1&0&-\frac{7}{6}&-\frac{473}{375}\\
 0&\frac{5}6&-\frac{1429}{2250}&-\frac{2999}{2250}\\
 \ast&\ast&5&6\\\ast&\ast&\ast&16\end{pmatrix},
 \tilde W=\begin{pmatrix}
  1&-\frac12&-1&\frac{31}{12}\\
  \ast&1&-\frac{19}{12}&\frac{679}{4500}\\
  \ast&\ast&\frac{8329}{1125}&-\frac{833}{100}\\
  \ast&\ast&\ast&\frac{69}2
 \end{pmatrix}.
\end{align*}
\par
This example may provide information of broader interest.
It shows that it may be possible to construct feasible weight matrices for parallel methods with four and maybe more stages even if the method is specified almost completely.
This indicates that the sufficient algebraic condition for A-stability from Theorem~\ref{TAAStb} may be rather sharp already with a constant, $\lambda$-independent weight matrix $Z$ as used in \eqref{Zdefnt}.
Also, the test matrix $\MM$ has rank 6 which is maximal according to Lemma~\ref{LKern} for $s=4$.
Rank defect two means that the Nordsieck version has an explicit representation $\tilde\MM=0_k\oplus\MM_D$ with $k=\lfloor(s+1)/2\rfloor$ where the nontrivial part $\MM_D$ is positive definite.
Definiteness of may be verified reliably in floating-point arithmetic. 
\begin{appendix}
\section{Detailed computations}
Space was saved in the 4 examples of the preceding sections by presenting the essential data of the constructed methods only.
In order to facilitate verification of the criterion the different transformations for all examples are reproduced in detail now.
\par\noindent{\bf Details for Example \ref{bspsq3}} 
The triangular coefficient matrix $G$ and the weight matrix $\hat Z$ have already been shown, the second weight matrix for \eqref{MMGen} is
\[ \hat W=
\begin{pmatrix}
\frac{76249}{76832}+\frac{63211}{691488}\sqrt{65}&
-\frac{33793}{115248}-\frac{33073}{499408}\sqrt{65}&
\frac{541}{4704}+\frac{2945}{183456}\sqrt{65}\\
\ast&
\frac{1887495}{3995264}+\frac{90585}{3995264}\sqrt{65}&
-\frac{16225}{244608}-\frac{661}{81536}\sqrt{65}\\
\ast&
\ast&
\frac{121}{4992}+\frac{37}{14976}\sqrt{65}
\end{pmatrix}
\]
and the coefficient matrix $B$ is given by
\[ B=\begin{pmatrix}
\frac{817}{1960}-\frac{9}{392}\sqrt{65}&
 \frac{369}{784}+\frac{675}{10192}\sqrt{65}&
 \frac{9}{80}-\frac{9}{208}\sqrt{65}\\
 \frac{1371}{6125}-\frac{39}{4900}\sqrt{65}&
 \frac{4001}{3920}-\frac{1917}{50960}\sqrt{65}&
 -\frac{489}{2000}+\frac{237}{5200}\sqrt{65}\\
 \frac{55737}{245000}-\frac{117}{9800}\sqrt{65}&
 \frac{13869}{19600}-\frac{12393}{254800}\sqrt{65}&
 \frac{649}{10000}+\frac{63}{1040}\sqrt{65}
\end{pmatrix}.
\]
The eigenvalues of $B$ are $\doteq1,0.4569,0.0455$ in accordance with preconsistency and zero stability.
By construction the test matrix $\hat\MM$ in \eqref{MMGen} has three zero blocks and only $\hat\MM_{22}=\hat W\succ0$ is definite.
Going back to the original formulation from Theorem~\ref{TAAStb} the transformed weight matrices are $H\T\hat ZH=$
\[ Z=
 \begin{pmatrix}
  \frac{3389263}{768320}+\frac{3518461}{6914880}\sqrt{65}&
  -\frac{902677}{460992}-\frac{577413}{1997632}\sqrt{65}&
  \frac{265}{672}+\frac{1415}{26208} \sqrt{65}\\
\ast&
  \frac{18279351}{7990528}+\frac{1358121}{7990528}\sqrt{65}&
  -\frac{34595}{69888}-\frac{1215}{23296}\sqrt{65}\\
\ast&
\ast&
  \frac{2825}{9984}+\frac{725}{29952}\sqrt{65}
 \end{pmatrix}
 \]
and $\Theta\T \hat W\Theta=$
\[
 W=\begin{pmatrix}
 \frac{1163}{10976}-\frac{767}{98784} \sqrt{65}&
 \frac{7585}{32928}-\frac{1647}{142688} \sqrt{65}&
 -\frac{1}{21}+\frac{29}{6552} \sqrt{65}\\
\ast&
 \frac{422145}{570752}+\frac{4383}{570752} \sqrt{65}&
 -\frac{7219}{34944}-\frac{15}{11648} \sqrt{65}\\
\ast&
\ast&
 \frac{1009}{4992}+\frac{253}{14976} \sqrt{65}
 \end{pmatrix}.
 \]
The smallest eigenvalues of $Z$ and $W$ are $\doteq0.204$ and $\doteq0.0167$.
The original test matrix $\MM$ \eqref{MDef} is a full matrix.
Maple identifies rank 3 exactly and the nontrivial eigenvalues in the interval $[0.127,3.033]$.
\par\noindent{\bf Details for Example \ref{bspsq4}} 
The coefficients of this diagonally-implicit 4-stage method are
\begin{align*}
G=&\begin{pmatrix}
0.5544770574&0&0&0\\
1.249724440&0.5544770574&0&0&\\
0.5816903621&-0.1016593761&0.5544770574&0\\
1.591284531&0.2616450399&-0.1745521833&0.5544770574
\end{pmatrix}\\
B=&\begin{pmatrix}
-0.1517881356&-0.191898050&1.405289413&-0.061603226\\
 0.423976329&0.56245777&0.163462193&-0.14989630\\
-0.2974468417&-0.421468111& 1.615228204&0.1036867507\\
 0.765740036&0.28293134&-0.26943806&0.22076669
\end{pmatrix}
\end{align*}
where $B$ has the eigenvalues $\doteq1,0.656,0.397,0.194$.
The difference matrix is
\[ E=\begin{pmatrix}
 -2.924587038&-2.026103023&4.236386571&0.714303489\\
 0.2475622961&-0.163685467&-1.071270925&0.9873940939\\
 -0.6719696591&1.390690288&-0.3042527174&-0.4144679131\\
 -0.391968580&-4.434411911&1.433855269&3.392525220
\end{pmatrix}.\]
This method was found by a numerical search procedure in the Nordsieck formulation with a given slack $\tilde\MM_{11}=\diag(0,0,1,2)$ having the exact rank defect from Lemma~\ref{lmspsd}.
Transforming the computed weight matrices back to the original form gives
\begin{align*}
Z=&\begin{pmatrix}
24.93687174&36.60191888&-41.10301520&-17.14663842\\
36.6019187&104.2243208&-93.94463411&-50.66305601\\
-41.1030151&-93.94463413&92.65849294&44.55523962\\
-17.14663839&-50.66305601&44.55523963&24.93746522
\end{pmatrix},\\
W=&\begin{pmatrix}
5.826768721&17.59843465&-15.32737112&-9.04101951\\
17.59843495&58.3547424&-47.28760605&-30.6386801\\
-15.32737123&-47.2876058&41.36332459&24.3072443\\
-9.04101962&-30.63867973&24.30724431&16.23315933
\end{pmatrix},
\end{align*}
with smallest eigenvalues $\doteq0.060$ and $\doteq0.038$.
With these matrices the test matrix $\MM$ in \eqref{MDef} is full and has 2 numerical eigenvalues of small magnitude indicating rank 6.
However, the smallest eigenvalues is negative, $-6\cdot 10^{-9}$.
This confirms that without the transformation \eqref{MThtr} semi-definiteness may not be verified reliably with real entries.
\par\noindent
The details of the reconstruction of this method from $\check E$ by the Algorithm are as follows.
Step~1 results in the matrix
\[\check H=\begin{pmatrix}
 1&0.9143733374&-0.8734437698&0.0868490964\\
 -0.9143733376&1.557385735&-0.5451184012&0.6055094088\\
 1.226883712&-1.704732078&2.767561341&0.2225164874\\
 0.1243293293&0.5657875861&-0.8827898483&1.889057453
\end{pmatrix}.
\]
Its triangular canonical form is $\check H=U_H L_H U_H^{-1}$ with
\begin{align*}
 L_H=&\begin{pmatrix}
 1.803501131&&&\\
 -2.878857141&1.803501131&&\\
 1.236357981&-0.9614900034&1.803501131&\\
 0.048570432&0.128230559&-0.235314677&1.803501131
 \end{pmatrix},\\
 U_H=&\begin{pmatrix}
 1&0.1516606939&-0.3187173423&0.5591253567\\
 &0.5468226337&0.4769330471&1.445134755\\
 &&0.9431116674&1.253027500\\
 &&&2.559774554
 \end{pmatrix},
\end{align*}
and for the transformed seed matrix we obtain
\[ E'=U_H^{-1}\check EU_H=\begin{pmatrix}
 0&1&-1.411974971&0.836862714\\
 &0&2&0.226595572\\ &&0&3\\&&&0
\end{pmatrix}
\]
satisfying the condition $2e'_{14}-e'_{13}(e'_{13}+e'_{24})\doteq0$ according to \eqref{LEL}.
We note that known values which were distorted by rounding errors have been replaced (i.e., $\check h_{11}$, the main diagonal of $L_H=H'$ and the superdiagonal of $E'$).
\par\noindent{\bf Details for Example \ref{bsppl3}} 
The coefficient matrices of this parallel 3-stage method are
\[ G=\begin{pmatrix}\frac25&\\ &\frac{20}{29}&\\&&\frac5{11}\end{pmatrix},
 \quad
 B=\begin{pmatrix}
\frac1{5}&-\frac15&1\\ -\frac1{29}&\frac{37}{29}&-\frac7{29}\\
 -\frac5{22}&\frac7{22}&\frac{10}{11}
\end{pmatrix}
\]
and $B$ has the eigenvalues $\doteq1,0.692\pm0.172 i$.
The original weight matrices are 
\begin{align*}
Z=H\T\hat ZH=&\begin{pmatrix}
 \frac{25}4& \frac{145}{24}&-\frac{55}4\\
\ast
 &\frac{841}{80}&-\frac{319}{20}\\
\ast&\ast
 &\frac{484}{15}
\end{pmatrix},\\
W=\Theta\T\hat W\Theta=&\frac1{12}\begin{pmatrix}
 37&59&-91\\ 
\ast
 &137&-167\\
\ast&\ast
 &236\end{pmatrix}
\end{align*}
with minimal eigenvalues $\doteq 0.122$ and $\doteq0.021$.
Since the test matrix $\hat\MM$ in \eqref{MDef} is a full matrix with rational entries Maple is able to find rank 3 with nontrivial eigenvalues $\in[0.12,40.16]$.
\par\noindent{\bf Details for Example \ref{bsppl4}} 
For this parallel 4-stage method the coefficient matrices are
\[ G=\frac1{25}\begin{pmatrix}
10&\\&16&\\&&24\\&&&30
\end{pmatrix},
\ B=\begin{pmatrix}
 -\frac2{15}&\frac{25}{21}&0&-\frac2{35}\\
 -\frac{31}{525}&\frac4{21}&\frac{52}{35}&-\frac{108}{175}\\
 \frac{31}{25}&-\frac{138}{35}&6&-\frac{402}{175}\\
 \frac{116}{35}&-\frac{135}{14}&\frac{165}{14}&-\frac{156}{35}
\end{pmatrix}
\]
and $B$ possesses the eigenvalues $1,3/5,\pm1/5$.
The weight matrices are
\begin{align*}
Z=H\T\hat ZH=&\begin{pmatrix}
 \frac{1299365}{63504}& -\frac{329675}{9072}&\frac{19656575}{762048}& -\frac{221965}{27216}\\
\ast
 &\frac{2551515625}{32514048}&-\frac{456171875}{6967296}& \frac{66676975}{3048192}\\
\ast&\ast
 &\frac{530265625}{8128512}&-\frac{719075}{31104}\\
\ast&\ast&\ast
 &\frac{1615015}{190512}
\end{pmatrix},\\
W=\Theta\T\hat W\Theta=&\begin{pmatrix}
\frac{743437}{79380}& -\frac{3233059}{127008}&\frac{1101553}{42336}&-\frac{530603}{52920}\\
\ast
&\frac{4662305}{63504}&-\frac{93605}{1176}&\frac{433471}{14112}\\
\ast&\ast
&\frac{1967785}{21168}&-\frac{509231}{14112}\\
\ast&\ast&\ast
&\frac{93346}{6615}
\end{pmatrix}
\end{align*}
with minimal eigenvalues $\doteq0.109$ and $\doteq0.062$.
With these matrices the test matrix \eqref{MDef} can be computed.
Its rank is 6 in accordance with the construction procedure and its nontrivial eigenvalues are contained in $[0.021,194.1]$.
\end{appendix}
\bibliographystyle{amsplain}
\newcommand{\sinum}{SIAM J. Numer. Anal.}
\newcommand{\sisc}{SIAM J. Sci. Comput.}

\end{document}